\documentclass[10pt,a4paper,twoside,leqno]{amsart}

\usepackage[english]{babel}
\usepackage[dvipsnames]{xcolor}
\definecolor{auburn}{rgb}{0.43, 0.21, 0.1}
\definecolor{blue(pigment)}{rgb}{0.2, 0.2, 0.6}
\definecolor{britishracinggreen}{rgb}{0.0, 0.26, 0.15}
\definecolor{cobalt}{rgb}{0.0, 0.28, 0.67}
\definecolor{ceruleanblue}{rgb}{0.16, 0.32, 0.75}
    \usepackage[utopia]{mathdesign}
    \DeclareSymbolFont{usualmathcal}{OMS}{cmsy}{m}{n}
    \DeclareSymbolFontAlphabet{\mathcal}{usualmathcal}
\usepackage[a4paper,top=3.5cm,bottom=3cm,left=3cm,
           right=3cm,bindingoffset=5mm]{geometry}
\usepackage{amsmath,amsthm}
\usepackage{indentfirst}
\usepackage[colorlinks,bookmarks]{hyperref} 
\hypersetup{colorlinks,%
citecolor=britishracinggreen,%
filecolor=black,%
linkcolor=cobalt,%
urlcolor=black}
\usepackage{braket,mathtools}
\usepackage{caption}
\usepackage{stmaryrd}
\usepackage{comment}
\usepackage{multirow}
\usepackage{microtype}

\newcommand*{\defeq}{\mathrel{\vcenter{\baselineskip0.5ex \lineskiplimit0pt
                     \hbox{\scriptsize.}\hbox{\scriptsize.}}}%
                     =}

\def\be{\begin{equation}}    
\def\ee{\end{equation}}
\def\bitem{\begin{itemize}}
\def\eitem{\end{itemize}}
\def\benum{\begin{enumerate}}
\def\eenum{\end{enumerate}}

\def\ra{\rightarrow}

\def\surj{\twoheadrightarrow} 

\def\Sch{\textrm{Sch}}

\def\Z{\mathbb Z}
\def\C{\mathbb C}

\def\O{\mathscr O}
\def\char{\textrm{char}}
\def\op{\textrm{op}}

\DeclareMathOperator{\Sets}{Sets}

\DeclareMathOperator{\Pic}{Pic}

\DeclareMathOperator{\Spec}{Spec\,}

\DeclareMathOperator{\dd}{d}

\DeclareMathOperator{\Hilb}{Hilb}

\DeclareMathOperator{\Sym}{Sym}
\DeclareMathOperator{\Aut}{Aut}
\DeclareMathOperator{\Ext}{Ext}

\DeclareMathOperator{\End}{End}

\DeclareMathOperator{\red}{red}
\DeclareMathOperator{\Spl}{Spl}

\makeatletter
\newenvironment{proofof}[1]{\par
  \pushQED{\qed}%
  \normalfont \topsep6\p@\@plus6\p@\relax
  \trivlist
  \item[\hskip3\labelsep
        \itshape
    Proof of #1\@addpunct{.}]\ignorespaces
}{%
  \popQED\endtrivlist\@endpefalse
}
\makeatother

\newtheoremstyle{note}
{3pt}
{3pt}
{}
{}
{\itshape}
{.}
{.5em}
{}
\theoremstyle{note}
\newtheorem*{conventions*}{Conventions}

\theoremstyle{definition}

\newtheorem*{lemma*}{Lemma}
\newtheorem*{theorem*}{Theorem}
\newtheorem*{example*}{Example}
\newtheorem*{fact*}{Fact}
\newtheorem*{notation*}{Notation}
\newtheorem*{definition*}{Definition}
\newtheorem*{prop*}{Proposition}
\newtheorem*{remark*}{Remark}
\newtheorem*{corollary*}{Corollary}

\newtheorem{definition}{Definition}[section]

\newtheorem{remark}[definition]{Remark}

\newtheoremstyle{thm} 
        {3mm}
        {3mm}
        {\slshape}
        {0mm}
        {\scshape}
        {.}
        {1mm}
        {}
\theoremstyle{thm}
\newtheorem{teo}[definition]{Theorem}
\newtheorem{cor}[definition]{Corollary}
\newtheorem{lemma}[definition]{Lemma}
\newtheorem{prop}[definition]{Proposition}
\newtheorem{theorem}{Theorem}

\usepackage{tikz}
\usepackage{tikz-cd}
\usepackage{rotating}

\usetikzlibrary{matrix,shapes,arrows,decorations.pathmorphing}
\tikzset{commutative diagrams/arrow style=math font}
\tikzset{commutative diagrams/.cd,
mysymbol/.style={start anchor=center,end anchor=center,draw=none}}
\newcommand\MySymb[2][\square]{%
  \arrow[mysymbol]{#2}[description]{#1}}

\numberwithin{equation}{section}

\title[Hilbert scheme of hyperelliptic Jacobians and Picard sheaves]{The Hilbert scheme of hyperelliptic Jacobians \\ and moduli of Picard sheaves}
\author[Andrea T. Ricolfi]{Andrea T. Ricolfi}
\address{Max Planck Institut f\"{u}r Mathematik}
\email{atricolfi@gmail.com}

\begin{document}

\maketitle

\begin{abstract}
Let $C$ be a hyperelliptic curve embedded in its Jacobian $J$ via an Abel--Jacobi map. We compute the scheme structure of the Hilbert scheme component of $\Hilb_J$ containing the Abel--Jacobi embedding as a point. We relate the result to the ramification (and to the fibres) of the Torelli morphism $\mathcal M_g\ra \mathcal A_g$ along the hyperelliptic locus. As an application, we determine the scheme structure of the moduli space of Picard sheaves (introduced by Mukai) on a hyperelliptic Jacobian.
\end{abstract}

\tableofcontents

\section{Introduction}

\subsection*{Main result}
In this paper we study the deformation theory of a smooth \emph{hyperelliptic} curve $C$ of genus $g\geq 2$, embedded in its Jacobian $J = (\Pic^0C,\Theta_C)$ via an Abel--Jacobi map
\[
\mathsf{aj}\colon C\hookrightarrow J.
\]
We work over an algebraically closed field $k$ of characteristic different from $2$. Our aim is to compute the scheme structure of the Hilbert scheme component 
\[
\Hilb_{C/J} \subset \Hilb_J
\]
containing the point defined by $\mathsf{aj}$. 
It is well known that the embedded deformations of $C$ into $J$ are parametrised by translations of $C$, and that they are \emph{obstructed} as long as $g\geq 3$ (see the next section for more details). In other words $\Hilb_{C/J}$ is \emph{singular}, with reduced underlying variety isomorphic to $J$. The tangent space dimension to the Hilbert scheme has been computed in \cite{LangeSernesi,Griffiths1}. The result is
\[
\dim_k H^0(C,N_C) = 2g-2.
\]
Therefore, as $\dim J = g$, the non-reduced structure of $\Hilb_{C/J}$ along $J$ is accounted for (up to first order) by $g-2$ extra tangents. By homogeneity of the Jacobian, it is natural to expect a decomposition
\[
\Hilb_{C/J} = J\times R_g
\]
for some artinian scheme $R_g$ with embedding dimension $g-2$. As we shall see, this is precisely what happens, and $R_g$ turns out to be the ``smallest'' (in the sense of Lemma \ref{lemma3gd82}) artinian scheme with the required embedding dimension. 
More precisely, let 
\be\label{rg}
R_g = \Spec k[s_1,\ldots,s_{g-2}]/\mathfrak m^2,
\ee
where $\mathfrak m = (s_1,\ldots,s_{g-2})$ is the maximal ideal of the origin. The main result of this paper (proved in Theorem \ref{main} in the main body) is the following.

\begin{theorem}\label{thm:thm1_hilb}
Let $C$ be a hyperelliptic curve of genus $g\geq 2$ over a field $k$ of characteristic different from $2$, and let $J$ be its Jacobian. 
Then there is an isomorphism of $k$-schemes
\[
\Hilb_{C/J} \,\cong\, J\times R_g,
\]
where $R_g$ is the artinian scheme \eqref{rg}.
\end{theorem}

\subsection*{Interpretation}
Let $\mathcal M_g$ be the moduli stack of smooth curves of genus $g$, and let $\mathcal A_g$ be the moduli stack of principally polarised abelian varieties of dimension $g$. The Torelli morphism
\[
\tau_g\colon \mathcal M_g\ra \mathcal A_g
\]
sends a curve $C$ to its Jacobian $J = \Pic^0C$, principally polarised by the Theta divisor $\Theta_C$. 
One can interpret the artinian scheme $R_g$ as the fibre of $\tau_g$
over a hyperelliptic point $[J,\Theta_C]\in \mathcal A_g$. This makes explicit the link between the \emph{ramification} of $\tau_g$ along the hyperelliptic locus (in other words, the failure of the infinitesimal Torelli property) and the singularities of the Hilbert scheme $\Hilb_{C/J}$ (in other words, the \emph{obstructions} to deform $C$ in $J$). We come back to this in Section \ref{remark:torelli_fibre}.

\subsection*{Moduli of Picard sheaves}
As an application of our result, in Section \ref{sec:picsheaves} we compute the scheme structure of certain moduli spaces of \emph{Picard sheaves} on a hyperelliptic Jacobian $J$. Mukai introduced these spaces as an application of his Fourier transform; he completed their study in the non-hyperelliptic case \cite{Mukai1,Mukai2}, leaving open the hyperelliptic one. 

Let $F$ be the Fourier--Mukai transform of a line bundle $\xi=\O_C(dp_0)$, where $p_0\in C$ and we assume $1\leq d\leq g-1$ to ensure that $F$ is a simple sheaf on $J$. Let $M(F)$ be the connected component of the moduli space of simple sheaves containing the point $[F]$. Mukai proved that $M(F)_{\red} = \widehat J\times J$, the isomorphism being given by the family of twists and translations of $F$ \cite[Example 1.15]{Mukai2}. Under the same assumptions of Theorem \ref{thm:thm1_hilb}, we prove the following (cf.~Theorem \ref{cor8183} in the main body of the text).

\begin{theorem}\label{thm:thm2_Picard}
There is an isomorphism of $k$-schemes 
\[
M(F)\,\cong\, \widehat J\times J\times R_g.
\]
\end{theorem}

\subsection*{Enumerative Geometry of abelian varieties}
A motivation for understanding the \emph{scheme structure} of classical moduli spaces such as the Hilbert scheme (Theorem \ref{thm:thm1_hilb}) and the moduli space of Picard sheaves (Theorem \ref{thm:thm2_Picard}) comes from the subject of Enumerative Geometry of abelian varieties. 

For instance, the Hilbert scheme of curves (in a $3$-fold) is the main player in Donaldson--Thomas theory --- see, for instance, \cite{BOPY16} for an exhaustive treatment (including several interesting conjectures) of the Enumerative Geometry of curves on abelian surfaces and $3$-folds. Understanding the scheme structure (or even the closed points!) of the Hilbert scheme of curves on a $3$-fold is very often a hopeless problem. Of course, Donaldson--Thomas theory has developed several sophisticated tools to deal with the lack of an explicit description of the Hilbert scheme; however, this paper shows that, at least for an arbitrary \emph{Abel--Jacobi curve}, the Hilbert scheme can be described completely. Thus an immediate corollary of Theorem \ref{thm:thm1_hilb} is the explicit description of the Donaldson--Thomas theory of an Abel--Jacobi curve, cf.~Section \ref{sec:DT}.

On the other hand, it is conceivable that the theory of \emph{Picard sheaves}, arising as a direct application of the Fourier--Mukai transform, could be exploited to aim for a deeper understanding of the intersection theory and cohomology of Jacobians, and possibly their compactifications. Having at one's disposal global results such as Theorem \ref{thm:thm2_Picard} might allow one to treat the \emph{whole} moduli space (the universal Jacobian over the moduli space of curves) at once in developing a theory of \emph{tautological rings} for (possibly compactified, universal) Jacobians, by combining Fourier--Mukai techniques with suitable analogues of the intersection theoretic calculations carried out in \cite{pagani2018pullbacks}.

\begin{conventions*}
We work over an algebraically closed field $k$ of characteristic $p\neq 2$. All curves are smooth and proper over $k$, they are (geometrically) connected, and their Jacobians are principally polarised by the Theta divisor.
\end{conventions*}

\section{Ramification of Torelli and the Hilbert scheme}

In this section we provide the framework where the problem tackled in this paper naturally lives in. 

\subsection{Deformations of Abel--Jacobi curves}
The following theorem was proved in the stated form by Lange--Sernesi, but see also the work of Griffiths \cite{Griffiths1}.

\begin{teo}[{\cite[Theorem 1.2]{LangeSernesi}}]\label{LSthm}
Let $C$ be a smooth curve of genus $g\geq 3$.
\bitem
\item [(i)] If $C$ is non-hyperelliptic, then $\Hilb_{C/J}$ is smooth of dimension $g$.
\item [(ii)] If $C$ is hyperelliptic, then $\Hilb_{C/J}$ is irreducible of dimension $g$ and everywhere non-reduced, with Zariski tangent space of dimension $2g-2$.
\eitem
In both cases, the only deformations of $C$ in $J$ are translations.
\end{teo}

The statement of Theorem \ref{LSthm} is proved over $\C$ in \cite{LangeSernesi}, but it holds over algebraically closed fields $k$ of arbitrary characteristic. To see this, we need Collino's extension of the Ran--Matsusaka criterion for Jacobians to an arbitrary field, which we state here for completeness.

\begin{teo}[\cite{Collino84}]\label{Coll}
Let $X$ be an abelian variety of dimension $g$ over an algebraically closed field $k$. Let $D$ be an effective $1$-cycle generating $X$ and let $\Theta\subset X$ be an ample divisor such that $D\cdot \Theta = g$. Then $(X,\Theta,D)$ is a Jacobian triple.
\end{teo}

\begin{proofof}{Theorem \ref{LSthm}}
Let $C\ra \Spec k$ be a smooth curve of genus $g$ and fix an Abel--Jacobi map $C\hookrightarrow J$. Consider the normal bundle exact sequence
\[
0\ra T_C\ra T_J|_C\ra N_C\ra 0.
\]
Since we have a canonical identification $T_J|_C = H^1(C,\O_C)\otimes_k\O_C$, the induced cohomology sequence is
\be\label{cohoJac}
0\ra H^1(C,\O_C)\ra H^0(C,N_C)\overset{\partial}{\ra} H^1(C,T_C)\overset{\sigma}{\ra} H^1(C,\O_C)^{\otimes 2}.
\ee
Since $H^0(C,N_C)$ is the tangent space to the Hilbert scheme, and $\dim_k H^1(C,\O_C) = g$, it is clear that $\Hilb_{C/J}$ is smooth of dimension $g$ if and only if $\partial = 0$, if and only if $\sigma$ is injective. The map $\sigma$ factors through the subspace $\Sym^2 H^1(C,\O_C)$, and its dual is the multiplication map
\[
\mu_C\colon \Sym^2 H^0(C,K_C)\ra H^0(C,K_C^2),
\]
where $K_C$ is the canonical line bundle of $C$. For a modern, fully detailed proof of the identification $\sigma^\vee = \mu_C$, we refer the reader to \cite[Theorem 4.3]{Landesman:2019aa}. By a theorem of Max Noether \cite[Chapter III \S~2]{ACGH}, the map $\mu_C$ is surjective if and only if $C$ is non-hyperelliptic (see also \cite{Griffiths1,Andreotti58} for different proofs). If $C$ is hyperelliptic, the quotient $H^0(C,N_C)/H^1(C,\O_C) = \textrm{Im }\partial$ has dimension $g-2$, as shown directly in \cite[Section 2]{OS1} by choosing appropriate bases of differentials. This proves part (i) of Theorem \ref{LSthm}, along with the count $h^0(C,N_C) = 2g-2$ (and the non-reducedness statement) of part (ii).
So in the non-hyperelliptic case, $\Hilb_{C/J}$ is smooth of dimension $g$.

To finish the proof of part (ii), suppose $C$ is hyperelliptic, and let $D\subset J$ be a closed $1$-dimensional $k$-subscheme defining a point of $\Hilb_{C/J}$. Then $D$ is represented by the \emph{minimal cohomology class}
\[
\frac{\Theta_C^{g-1}}{(g-1)!}
\]
on $J$. This implies at once that $D$ generates $J$, and that $D\cdot \Theta_C = g$. Therefore, by Theorem \ref{Coll}, $(\Pic^0D,\Theta_D)$ and $(J,\Theta_C)$ are isomorphic as principally polarised abelian varieties. By Torelli's theorem, this implies (using also that $C$ is hyperelliptic) that $D$ is a translate of $C$. Thus $\Hilb_{C/J}$ is irreducible of dimension $g$, and its $k$-points coincide with those of $J$. The result follows.
\end{proofof}

\begin{cor}\label{cor:non_hyperelliptic_iso}
Let $J$ be the Jacobian of a non-hyperelliptic curve $C$. Then the family of translations of $C$ inside $J$ induces an isomorphism
\[
J \,\,\widetilde{\to}\,\, \Hilb_{C/J}.
\]
\end{cor}

\begin{proof}
The natural morphism $h\colon J \to \Hilb_{C/J}$ is proper (since $J$ is proper and the Hilbert scheme is proper, hence separated), 
injective on points and tangent spaces --- since the tangent map at $0 \in J$ is the map $\dd h\colon H^1(C,\O_C)\hookrightarrow H^0(C,N_C)$ in the sequence \eqref{cohoJac}.
Thus $h$ is a closed immersion, in particular it is unramified. 
However, the proof of Theorem \ref{LSthm} shows that $h\colon J \to \Hilb_{C/J}$ is bijective and, since $C$ is non-hyperelliptic, $\dd h$ is an isomorphism. Thus $h$ is an isomorphism.
\end{proof}


\begin{remark}
If $C$ is a generic complex curve of genus at least $3$, its $1$-cycle on $J$ is not algebraically equivalent to the cycle of $-C$ by a famous theorem of Ceresa \cite{Ceresa1}. Here $-C$ is the image of $C$ under the automorphism $-1\colon J\ra J$. Therefore the Hilbert scheme $\Hilb_J$ contains another component $\Hilb_{-C/J}$, disjoint from $\Hilb_{C/J}$ and still isomorphic to $J$.
\end{remark}

\subsection{Torelli problems}\label{gnagnagne}

Consider the Torelli morphism
\[
\tau_g\colon \mathcal M_g\ra \mathcal A_g
\]
from the stack of nonsingular curves of genus $g$ to the stack of principally polarised abelian varieties, sending a curve to its (canonically polarised) Jacobian. The \emph{infinitesimal Torelli problem} asks whether the Torelli morphism is an immersion. It is well known that $\tau_g$ is ramified along the hyperelliptic locus: this is again Noether's theorem, stating that $\mu_C$, the \emph{codifferential} of $\tau_g$ at $[C]\in \mathcal M_g$, is not surjective. So, even though $\tau_g$ is injective on geometric points by Torelli's theorem, it is not an immersion.

To sum up, we have the following. Let $C$ be an arbitrary smooth curve of genus $g\geq 3$, and let $J$ be its Jacobian. Then the following conditions are equivalent:
\bitem
\item [(i)] $C$ is hyperelliptic,
\item [(ii)] $\Hilb_{C/J}$ is singular at $[\mathsf{aj}\colon C\hookrightarrow J]$,
\item [(iii)] the embedded deformations of $C$ into $J$ are obstructed,
\item [(iv)] $\tau_g\colon \mathcal M_g\ra \mathcal A_g$ is ramified at $[C]$,
\item [(v)] infinitesimal Torelli fails at $C$.
\eitem

\medskip
The \emph{local Torelli problem} for curves, studied by Oort and Steenbrink in \cite{OS1}, asks whether the morphism
\[
t_g\colon M_g\ra A_g
\]
between the coarse moduli spaces is an immersion. These schemes do not represent the corresponding moduli functors, so the local structure of $t_g$ is not (directly) linked with deformation theory of curves and their Jacobians. However, introducing suitable level structures, one replaces the normal varieties $M_g$ and $A_g$ with smooth varieties 
\[
M_g^{(n)}, \quad A_g^{(n)}
\]
that are \emph{fine} moduli spaces for the corresponding moduli problem, and are \'etale over $\mathcal M_g$ and $\mathcal A_g$, respectively.

Let $p\geq 0$ be the characteristic of the base field. Oort and Steenbrink show that $t_g$ is an immersion if $p = 0$. The answer to the local Torelli problem is also affirmative if $p > 2$, at almost all points of $M_g$. More precisely, $t_g$ is an immersion at those points in $M_g$ representing curves $C$ such that $\Aut C$ has no elements of order $p$ \cite[Cor.~3.2]{OS1}. Finally, $t_g$ is \emph{not} an immersion if $p=2$ and $g\geq 5$ \cite[Cor.~5.3]{OS1}.

\section{Moduli spaces with level structures}
In this section we introduce the moduli spaces of curves and abelian varieties we will be working with throughout.

\subsection{Level structures}
Let $S$ be a scheme. An abelian scheme over $S$ is a group scheme $X\ra S$ which is smooth and proper and has geometrically connected fibres. We let $\widehat X\ra S$ denote the dual abelian scheme. A polarisation on $X\ra S$ is an $S$-morphism $\lambda\colon X\ra \widehat X$ such that its restriction to every geometric point $s\in S$ is of the form 
\[
\phi_{\mathscr L}\colon X_s\ra \widehat X_s, \quad x\mapsto \mathsf t_{x}^\ast \mathscr L\otimes \mathscr L^\vee,
\]
for some ample line bundle $\mathscr L$ on $X_s$. Here and in what follows, $\mathsf t_x$ is the translation $y\mapsto x+y$ by the element $x\in X_s$. We say $\lambda$ is \emph{principal} if it is an isomorphism.

Fix an integer $n>0$ and an abelian scheme $X\ra S$ of relative dimension $g$. Multiplication by $n$ is an $S$-morphism of group schemes
\[
[n] \colon X\ra X,
\]
and we denote its kernel by $X_n$. Assuming $n$ is not divisible by $p$, we have that $X_n$ is an \'etale group scheme over $S$, locally isomorphic in the \'etale topology to the constant group scheme $(\Z/n\Z)^{2g}$. One has $\widehat X_n = X_n^D$, where the superscript $D$ denotes the Cartier dual of a finite group scheme. Then any principal polarisation $\lambda$ on $X$ induces a skew-symmetric bilinear form
\[
E_n \colon X_n\times_S X_n\xrightarrow{\textrm{id}\times \lambda} X_n\times_S X_n^D\xrightarrow{e_n} \mu_{n},
\]
where $e_n$ is the Weil pairing. The group $\Z/n\Z$ is Cartier dual to $\mu_n$. We endow $(\Z/n\Z)^{g}\,\widetilde{\ra}\,\mu_n^g$ with the standard symplectic structure, given by the $2g\times 2g$ matrix
\[
\begin{pmatrix}
0 & \mathbb 1_g \\
-\mathbb 1_g & 0
\end{pmatrix}.
\]

\begin{definition}[\cite{OS1}]
A (symplectic) level-$n$ structure on a principally polarised abelian scheme $(X/S,\lambda)$ is a symplectic isomorphism 
\[
\alpha\colon (X_n,E_n)~\widetilde{\ra}~ (\Z/n\Z)^{2g}.
\]
A level-$n$ structure on a smooth proper curve $\mathcal C\ra S$ is a level structure on its Jacobian $\Pic^0(\mathcal C/S)\ra S$.
\end{definition}

Curves with level structure are represented by pairs $(C,\alpha)$. We consider $(C,\alpha)$ and $(C',\alpha')$ as being isomorphic if there is an isomorphism $u\colon C\,\widetilde{\ra}\,C'$ such that the induced isomorphism $J(u)\colon J'\,\widetilde{\ra}\,J$ between the Jacobians takes $\alpha'$ to $\alpha$. An isomorphism between $(X,\lambda,\alpha)$ and $(X',\lambda',\alpha')$ is an isomorphism $(X',\lambda')\,\widetilde{\ra}\,(X,\lambda)$ of principally polarised abelian schemes, taking $\alpha'$ to $\alpha$.

\begin{remark}\label{notisom382}
If $C$ is a curve of genus $g\geq 3$ with trivial automorphism group, and $\alpha$ is a level structure on $C$, then $(C,\alpha)$ is not isomorphic to $(C,-\alpha)$. On the other hand, if $J$ denotes the Jacobian of $C$, one has that $(J,\Theta_C,\alpha)$ and $(J,\Theta_C,-\alpha)$ are isomorphic, because the automorphism $-1\colon J\ra J$, defined globally on $J$, identifies the two pairs.
\end{remark}

\subsubsection{Choice of level}\label{assum}
As indicated by Theorem \ref{thm:fine_moduli} below, moduli spaces of curves and abelian varieties with level structure are well behaved when the condition $(p,n)=1$ is met.
For later purposes, we need to strengthen the condition $(p,n)=1$.
Note that $p = \char\, k$ is fixed, as well as the genus $g$. However, we are free to choose $n\geq 3$, and the condition we require
is that the order of the symplectic group
\[
\lvert \textrm{Sp}(2g,\Z/n\Z)\rvert = n^{g^2}\cdot \prod_{i=1}^g\,(n^{2i}-1)
\]
is not divisible by $p$. In particular, this implies $(p,n) = 1$. From now on, 
\begin{equation}\label{choice_of_n}
n\textrm{ is fixed in such a way that }p\textrm{ does not divide }\lvert \textrm{Sp}(2g,\Z/n\Z)\rvert.
\end{equation}
This condition implies that the symplectic group $\textrm{Sp}(2g,\Z/n\Z)$ acts freely and transitively on the set of symplectic level-$n$ structures on a smooth curve defined over $k$. This will be used in the proof of Lemma \ref{lemma:forget}.

\subsection{Moduli spaces}
Let $\mathscr M_g^{(n)}$ be the functor $\Sch_k^{\op}\ra \Sets$ sending a $k$-scheme $S$ to the set of $S$-isomorphism classes of curves of genus $g$ with level-$n$ structure. Similarly, let $\mathscr A_g^{(n)}$ be the functor sending $S$ to the set of $S$-isomorphism classes of principally polarised abelian schemes of relative dimension $g$ over $S$ equipped with a level-$n$ structure.

\begin{teo}\label{thm:fine_moduli}
If $n\geq 3$ and $(p,n) = 1$, the functors $\mathscr M_g^{(n)}$ and $\mathscr A_g^{(n)}$ are represented by smooth quasi-projective varieties $M_g^{(n)}$ and $A_g^{(n)}$ of dimensions $3g-3$ and $g(g+1)/2$ respectively.
\end{teo}

\begin{proof}
For the statement about $\mathscr M_g^{(n)}$ we refer to \cite{Popp}, whereas the one about $\mathscr A_g^{(n)}$ is \cite[Theorem 7.9]{Mumford65}.
\end{proof}

Consider the morphism
\be\label{jn22}
j_n \colon M_g^{(n)}\ra A^{(n)}_g
\ee
sending a curve with level structure to its Jacobian, as usual principally polarised by the Theta divisor. The map $j_n$ is generically of degree two onto its image, essentially because of Remark \ref{notisom382}.
To link it back to $t_g\colon M_g\ra A_g$, Oort and Steenbrink form the geometric quotient
\[
V^{(n)} = M_g^{(n)}/\Sigma,
\]
where 
\be\label{sigmainv}
\Sigma\colon M_g^{(n)}\ra M_g^{(n)}
\ee
is the involution sending $[D,\beta]\mapsto [D,-\beta]$. Note that $\Sigma$ is the identity if $g \leq 2$.
The map $j_n$ factors through a morphism
\[
\iota \colon V^{(n)}\ra A^{(n)}_g,
\]
which turns out to be injective on geometric points \cite[Lemma 1.11]{OS1}. In fact, we need the following stronger statement.

\begin{teo}[{\cite[Theorem 3.1]{OS1}}]\label{thm:OSemb}
If $g\geq 2$ and $\mathrm{char}\, k \neq 2$ then $\iota$ is an immersion.
\end{teo}

Oort and Steenbrink use this result crucially to solve the local Torelli problem as we recalled in Section \ref{gnagnagne}.
For us, it is not important to have the statement of local Torelli (which strictly speaking only holds globally in characteristic $0$): all we need in our argument is Theorem \ref{thm:OSemb}, which is why we require the base field $k$ to have characteristic $p\neq 2$.

The following result was proven in \cite[Prop.~5.8]{DMstablecurves} in greater generality. We give a short proof here for the sake of completeness.

\begin{lemma}\label{lemma:forget}
The maps $\varphi\colon M_g^{(n)}\ra \mathcal M_g$ and $\psi\colon A_g^{(n)}\ra \mathcal A_g$ forgetting the level structure are \'etale.
\end{lemma}

\begin{proof}
We start by showing that $\varphi$ is flat. Choose an atlas for $\mathcal M_g$, that is, an \'etale surjective map $a\colon U\ra \mathcal M_g$ from a scheme. Form the fibre square
\[
\begin{tikzcd}
V\MySymb{dr}\arrow{r}{b}\arrow[swap]{d}{} & M_g^{(n)} \arrow{d}{\varphi} \\
U\arrow[swap]{r}{a} & \mathcal M_g
\end{tikzcd}
\]
and pick a point $u\in U$, with image $y = a(u)\in \mathcal M_g$. The fibre $V_u\subset V$ is contained in $b^{-1}\varphi^{-1}(y)$, which is \'etale over $\varphi^{-1}(y)$ because $b$ is \'etale. In particular, since $\varphi^{-1}(y)$ is finite, the same is true for $V_u$. Therefore $V\ra U$ is a map of smooth varieties with fibres of the same dimension (zero); by ``miracle flatness'' \cite[Prop.~15.4.2]{MR0217086}, it is flat; therefore $\varphi$ is flat. On the other hand, the geometric fibres of $\varphi$ are the symplectic groups $\textrm{Sp}(2g,\Z/n\Z)$, and they are reduced by our choice of $n$ (cf.~\eqref{choice_of_n} in Section~\ref{assum}). Hence $\varphi$ is smooth of relative dimension zero, that is, \'etale. The same argument applies to the map $\psi$, with the symplectic group replaced by $\textrm{Sp}(2g,\Z/n\Z)/\pm \mathbb 1$.
\end{proof}

\begin{remark}
The maps $M_g^{(n)}\ra M_g$ and $A_g^{(n)}\ra A_g$ down to the coarse moduli schemes are still finite Galois covers, but they are not \'etale.
\end{remark}

By Lemma \ref{lemma:forget}, we can identity the tangent space to a point $[C,\alpha]\in M_g^{(n)}$ with the 
tangent space to its image $[C]\in \mathcal M_g$ under $\varphi$, and similarly on the abelian variety side. Moreover, the cartesian diagram
\be\label{square734}
\begin{tikzcd}
M_g^{(n)}\MySymb{dr}\arrow{r}{j_n}\arrow[swap]{d}{\varphi} & A_g^{(n)}\arrow{d}{\psi}\\
\mathcal M_g\arrow[swap]{r}{\tau_g} & \mathcal A_g
\end{tikzcd}
\ee
allows us to identify the map 
\[
\sigma\colon H^1(C,T_C)\ra \Sym^2H^1(C,\O_C),
\]
already appeared in \eqref{cohoJac}, with the tangent map of $j_n$ at a point $[C,\alpha]$.
As we already mentioned, in \cite[Section 2]{OS1} it is shown that if $C$ is hyperelliptic the kernel of $\sigma$ has dimension $g-2$. 

\section{Proof of the main theorem}

\subsection{Proof of Theorem \ref{thm:thm1_hilb}}
Let $C$ be a hyperelliptic curve of genus $g\geq 3$ and let $J$ be its Jacobian. Fix an Abel--Jacobi embedding $C\hookrightarrow J$ and let
\[
H \defeq \Hilb_{C/J}
\]
be the Hilbert scheme component containing such embedding as a point.
Let 
\[
\begin{tikzcd}
\mathcal Z\arrow[hook]{r}{\iota}\arrow{d} & H\times J\arrow{dl}{\textrm{pr}_1} \\
H &
\end{tikzcd}
\]
be the universal family over the Hilbert scheme.

\begin{lemma}\label{lemmaisoab}
The restriction morphism 
\[
\iota^\ast\colon \Pic^0(H\times J/H)\ra \Pic^0(\mathcal Z/H) 
\]
is an isomorphism of abelian schemes over $H$.
\end{lemma}

\begin{proof}
We use the \emph{crit\`ere de platitude par fibres} \cite[Th\'eor\`eme 11.3.10]{MR0217086} in the following special case: suppose given a scheme $S$ and an $S$-morphism $f\colon X\ra Y$ such that: (a) $X/S$ is finitely presented and flat, (b) $Y/S$ is locally of finite type, and (c) $f_s\colon X_s\ra Y_s$ is flat for each $s\in S$. Then $f$ is flat.
Applying this to $(S,f) = (H,\iota^\ast)$, we conclude that $\iota^\ast$ is flat. But $\Pic^0(H\times J/H)$ is isomorphic, over $H$, to the constant abelian scheme $H\times J$, and $\iota^\ast$ is an isomorphism on each fibre over $H$. Therefore it is a flat, unramified and bijective morphism, hence an isomorphism.
\end{proof}

Let $\alpha$ be a fixed level-$n$ structure on $J$, with $n\geq 3$ chosen as in Section \ref{assum}.
Form the constant level structure $\alpha_H$ on the abelian scheme $H\times J\ra H$.
Transferring the level structure $\alpha_H$ from $H\times J$ to $\Pic^0(\mathcal Z/H)$ using the isomorphism $\iota^\ast$ of Lemma \ref{lemmaisoab}, we can now regard $\mathcal Z\ra H$ as a family of Abel--Jacobi curves with level-$n$ structure. Since $M_g^{(n)}$ is a \emph{fine} moduli space for these objects, we obtain a morphism
\be\label{mapf}
f\colon H \ra M_g^{(n)}.
\ee
Note that the topological image of $f$ is just the point $x\in M_g^{(n)}$ corresponding to $[C,\alpha]$. The tangent map $\dd f$ at the point $[C]\in H$ is the connecting homomorphism
\[
\partial\colon H^0(C,N_C)\ra H^1(C,T_C),
\]
already appeared in \eqref{cohoJac}.

Our next goal is to view the Hilbert scheme $H$ over a suitable artinian scheme $R_g$. Recall the Torelli type morphism $j_n$ introduced in \eqref{jn22}.
We define 
\[
R_g\subset M_g^{(n)}
\]
to be the scheme-theoretic fibre of $j_n$ over the moduli point $[J,\alpha]\in A_g^{(n)}$.
Let $y\in V^{(n)}$ be the image of the point $x = [C,\alpha]$ under the quotient map
\[
M_g^{(n)}\ra V^{(n)} = M_g^{(n)}/\Sigma,
\]
where $\Sigma$ is the involution first appeared in \eqref{sigmainv}.
During the proof of \cite[Cor.~3.2]{OS1} it is shown that one can choose local coordinates $t_1,\dots,t_{3g-3}$ around $x$ such that 
$\Sigma^\ast t_i = t_i$ if $i=1,\dots,2g-1$ and $\Sigma^\ast t_i = -t_i$ if $i=2g,\dots,3g-3$.  Oort--Steenbrink deduce that
\be\label{localstructure}
\widehat{\O}_{y} = \widehat{\O}_{x}^{\,\,\Sigma} = k\llbracket t_1, \dots, t_{2g-1}, t_{2g}^2,t_{2g}t_{2g+1},\dots,t_{3g-3}^2\rrbracket.
\ee
Since we have a factorisation 
\[
j_n\colon M_g^{(n)}\ra V^{(n)}\overset{\iota}{\hookrightarrow} A_g^{(n)}
\]
where $\iota$ is an \emph{immersion} by Theorem \ref{thm:OSemb}, we deduce from \eqref{localstructure} that
\[
R_g = \Spec k[s_1,\dots,s_{g-2}]/\mathfrak m^2,
\]
where $\mathfrak m = (s_1,\dots,s_{g-2})\subset k[s_1,\dots,s_{g-2}]$.
For instance, $R_3$ is the scheme of dual numbers $k[s]/s^2$, and if $g=4$ we get the triple point $k[s,t]/(s^2,st,t^2)$.



Recall the cohomology sequence 
\be\label{longexseq}
0\ra H^1(C,\O_C)\ra H^0(C,N_C)\overset{\partial}{\ra} H^1(C,T_C)\overset{\sigma}{\ra}H^1(C,\O_C)^{\otimes 2},
\ee
where $\sigma$ factors through $\Sym^2 H^1(C,\O_C)$, the tangent space of $\mathcal A_g$ at $[J,\Theta_C]$. Since $C$ is hyperelliptic, the image of $\partial$ has dimension $g-2 > 0$. In other words, the differential $\partial = \dd f$, where $f$ was defined in \eqref{mapf}, does not vanish at the point $[C]\in H$. 
Thus $f$ is not scheme-theoretically constant, although $x = [C,\alpha]\in M_g^{(n)}$ is the only point in the image. On the other hand, the composition
\[
j_n\circ f\colon H\ra M_g^{(n)} \ra A^{(n)}_g
\]
is the constant morphism since its differential is identically zero. 
Indeed the composition
\[
\sigma\circ\partial\colon H^0(C,N_C)\ra H^1(C,T_C)\ra \Sym^2 H^1(C,\O_C)
\]
vanishes by exactness of \eqref{longexseq}. So the image point $[J,\alpha]$ does not deform even at first order, and we conclude that $f$ factors through the scheme-theoretic fibre of $j_n$. This gives us a morphism
\be\label{hilbsuartin}
\pi\colon H\ra R_g.
\ee

We will exploit the following technical lemma.

\begin{lemma}[{\cite[Lemma 1.10.1]{kollar1}}]\label{lemmakollar}
Let $R$ be the spectrum of a local ring, $p\colon U\ra V$ a morphism over $R$, with $U\ra R$ flat and proper. If the restriction $p_0\colon U_0\ra V_0$ of $p$ over the closed point $0\in R$ is an isomorphism, then $p$ is an isomorphism.
\end{lemma}

Recall that $J = H_{\red}$, so we have a closed immersion $J\hookrightarrow H$ (with empty complement). Consider the closed point $0\in J$ corresponding to $C$. 
Let us fix a regular sequence $f_1,\ldots,f_g$ in the maximal ideal of $\O_{J,0}$. Choose lifts $\widetilde f_i\in \O_{H,0}$ along the natural surjection $\O_{H,0}\surj \O_{J,0}$, for $i = 1,\ldots, g$. Then we consider the zero scheme
\begin{equation}\label{def:S_g}
i\colon S_g = Z(\widetilde f_1,\ldots,\widetilde f_g)\hookrightarrow H,
\end{equation}
the largest artinian scheme supported at $0\in H$.
We next show that the composition
\be\label{rhomap}
\rho = \pi\circ i\colon S_g\hookrightarrow H\ra R_g
\ee
is an isomorphism, where $\pi$ is defined in \eqref{hilbsuartin}.
We will need the following lemma.

\begin{lemma}\label{lemma3gd82}
Let $\ell\colon k[x_1,\ldots,x_d]/\mathfrak m^2\surj B$ be a surjection of local Artin $k$-algebras such that the differential $\dd \ell$ is an isomorphism. Then $\ell$ is an isomorphism.
\end{lemma}

\begin{proof}
Since  $\dd \ell$ is an isomorphism by assumption, $B$ has embedding dimension $d$, hence it can be written as a quotient $k[x_1,\ldots,x_d]/I$, so that its maximal ideal is $\mathfrak m_B = \mathfrak m/I$. Starting from the surjection $\ell$, it is then clear that $\mathfrak m^2 \subset I$, and we have to show the other inclusion. This follows from the chain of isomorphisms
\[
\mathfrak m/\mathfrak m^2 \,\,\widetilde{\to}\,\,\mathfrak m_B/\mathfrak m_B^2 
= \frac{\mathfrak m/I}{(\mathfrak m/I)^2} 
= \frac{\mathfrak m/I}{\mathfrak m^2/I\cap \mathfrak m^2} 
= \frac{\mathfrak m/\mathfrak m^2}{I / \mathfrak m^2},
\]
where the first isomorphism is $(\dd \ell)^\vee$.
\end{proof}

\begin{lemma}\label{lemma327637}
The tangent map $\dd \rho\colon T_{S_g}\ra T_{R_g}$ is an isomorphism.
\end{lemma}

\begin{proof}
The kernel of $H^1(C,T_C)\ra H^1(C,\O_C)^{\otimes 2}$, namely the image of $\partial\colon H^0(C,N_C)\ra H^1(C,T_C)$, is the tangent space $T_{R_g}$ to the artinian scheme $R_g$, as the latter is by definition the fibre of $j_n$.
We then have a direct sum decomposition $T_0H = T_0J\oplus T_{R_g}$. 
The intersection of $S_g$ and $J$ inside $H$ is the reduced origin $0\in J$, so 
the linear subspace $T_{S_g}\subset T_0H$ intersects $T_0J$ trivially, which implies that the tangent map
\[
\dd\rho\colon T_{S_g}\subset T_0J\oplus T_{R_g}\ra T_{R_g}
\]
is injective. On the other hand, the inclusion $T_{S_g}\subset T_0H$ is cut out by independent linear functions, again because $T_{S_g}\cap T_0J = (0)$. It follows that the linear inclusion $T_{S_g}\subset T_0H$ has codimension equal to $\dim T_0J = g$, thus
\[
\dim T_{S_g} = \dim T_0H - g = g-2 = \dim T_{R_g}.
\]
The result follows.
\end{proof}


\begin{cor}\label{corhuqweqe}
The map $\rho\colon S_g\ra R_g$ of \eqref{rhomap} is an isomorphism.
\end{cor}

\begin{proof}
The map $\rho$ is proper, injective on points and, by Lemma \ref{lemma327637}, injective on tangent spaces. Then it is a closed immersion; in fact, by Lemma \ref{lemma327637} again, it is an isomorphism on tangent spaces, so by Lemma \ref{lemma3gd82} it is an isomorphism.
\end{proof}

The corollary yields a section of $\pi$,
\[
s = i\circ \rho^{-1}\colon R_g\, \widetilde{\ra}\, S_g\hookrightarrow H,
\]
which finally allows us to prove the main result of this paper.

\begin{teo}\label{main}
Let $C$ he a hyperelliptic curve of genus $g\geq 2$, and let $J$ be its Jacobian. Then there is an isomorphism of schemes
\[
J\times R_g\,\,\widetilde{\ra}\,\, H.
\]
\end{teo}

\begin{proof}
If $g = 2$, the Hilbert scheme is nonsingular because $\partial\colon H^0(C,N_C)\ra H^1(C,T_C)$, the connecting homomorphism  in \eqref{cohoJac}, vanishes. If $g\geq 3$, consider the translation action $\mu\colon J\times H\ra H$ by $J$ on the Hilbert scheme and the composition
\[
J\times R_g\xhookrightarrow{\textrm{id}_J\times s} 
J\times H\xrightarrow{\mu} 
H,
\]
viewed as a morphism over the artinian scheme $R_g$. Since it restricts to the identity $\textrm{id}_J$ over the closed point of $R_g$, by Lemma \ref{lemmakollar} it must be an isomorphism.
\end{proof}

\subsection{Relation between Hilbert scheme and Torelli}\label{remark:torelli_fibre}
Let $z = [J,\Theta_C]$ be a point in the image of the Torelli morphism $\tau_g\colon \mathcal M_g \to \mathcal A_g$. The fibre of $\tau_g$ over $\Spec k(z) \to \mathcal A_g$ is, topologically, just a point, by Torellli's theorem. This point is scheme-theoretically reduced if $C$ is non-hyperelliptic. However, thanks to the cartesian diagram \eqref{square734}, what we can observe is that $\tau_g^{-1}(z) = \mathcal M_g \times_{\mathcal A_g} \Spec k(z) \subset \mathcal M_g$ is the artinian 
scheme $R_g$ when $z$ represents a hyperelliptic Jacobian. Theorem \ref{main} thus fully develops in a qualitative form the idea already present in \cite{LangeSernesi}, namely that understanding the ramification (the fibres) of the Torelli morphism is equivalent to understanding the singularities of the Hilbert scheme: what the present work shows is that these singularities are controlled by the artinian scheme $R_g$.

The results proved so far essentially show the following.

\begin{prop}
Let $C$ be a smooth curve of genus $g\geq 2$, and let $J$ be its Jacobian. 
Then $\tau_g^{-1}([J,\Theta_C])$ is isomorphic to the largest closed subscheme of $\Hilb_{C/J}$ supported at $[\mathsf{aj}\colon C \hookrightarrow J]$.
\end{prop}

\begin{proof}
In the non-hyperelliptic case, we have $\tau_g^{-1}([J,\Theta_C]) \cong \Spec k$, because $\tau_g$ is unramified at $[C]$. The result then follows because $J\to \Hilb_{C/J}$ is an isomorphism (by Corollary \ref{cor:non_hyperelliptic_iso}). In the hyperelliptic case we get, using Lemma \ref{corhuqweqe},
\[
S_g \,\,\widetilde{\to}\,\, R_g = \tau_g^{-1}([J,\Theta_C]),
\]
where $S_g \subset \Hilb_{C/J}$, introduced in \eqref{def:S_g}, is precisely the largest 
subscheme of the Hilbert scheme supported at $[\mathsf{aj}\colon C \hookrightarrow J]$.
\end{proof}

\subsection{Donaldson--Thomas invariants for Jacobians}\label{sec:DT}
Let $C$ be a smooth complex projective curve of genus $3$. One can study the ``$C$-local Donaldson--Thomas invariants'' of the abelian $3$-fold $J=\Pic^0C$. As explained in \cite{LocalDT,Ricolfi2018}, these invariants are completely determined by the ``BPS number'' of the curve,
\[
n_C = \nu_H(\mathscr I_C)\in \Z,
\]
in the sense that their generating function is equal to the rational function
\[
n_C\cdot q^{-2}(1+q)^{4}.
\]
Here $\nu_H\colon \Hilb_{C/J}\ra \Z$ is the Behrend function of the Hilbert scheme. The Behrend function attached to a general finite type $\C$-scheme $X$ is an invariant of the singularities of $X$. It was introduced in \cite{Beh} and is now a key tool in Donaldson--Thomas theory. For a smooth scheme $Y$ one has that $\nu_Y$ is the constant $(-1)^{\dim Y}$, and moreover $\nu_{X\times Y} = \nu_X\cdot \nu_Y$ for two complex schemes $X$ and $Y$.
While for non-hyperelliptic $C$ we have $n_C=-1$ (because the Hilbert scheme is a copy of the smooth $3$-fold $J$), the structure result
\[
\Hilb_{C/J} = J\times \Spec \C[s]/s^2
\]
in the hyperelliptic case yields $n_C = -2$, because the scheme of dual numbers has Behrend function $\nu_{R_3} = 2$.

\section{An application to moduli spaces of Picard sheaves}\label{sec:picsheaves}

Mukai introduced in \cite{Mukai1} his celebrated Fourier transform, and gave an application to the moduli space of Picard sheaves on Jacobians of curves. We now review his results on non-hyperelliptic Jacobians and extend them to the hyperelliptic case. We assume that the base field $k$ is, as ever, algebraically closed of characteristic different from $2$.

We let $\Phi\colon \mathrm D^b(\widehat J)\ra \mathrm D^b(J)$ be the Fourier transform with kernel the Poincar\'e line bundle $\mathscr P\in \Pic(\widehat J\times J)$. 
If $\widehat{\mathsf p}\colon\widehat J\times J\ra \widehat J$ and $\mathsf p\colon\widehat J\times J\ra J$ are the projections, by definition one has
\[
\Phi(\mathscr E) = R\mathsf p_\ast(\widehat{\mathsf p}^\ast \mathscr E\otimes \mathscr P).
\]
We will denote by $\Phi^i(\mathscr E)$ the $i$-th cohomology sheaf of the complex $\Phi(\mathscr E)$.

Let $p_0\in C$ be a point on a smooth curve of genus $g\geq 2$. Let us form the line bundle $\xi = \O_C(dp_0)$. From now on we view it as a sheaf on $\widehat J$ by pushing it forward along the Abel--Jacobi map $\mathsf{aj}\colon C\hookrightarrow J$ followed by the identification of $J$ with its dual. 
Applying his Fourier transform, Mukai constructs
\be\label{simpleF}
F = \Phi^1(\mathsf{aj}_\ast\xi),
\ee
a \emph{Picard sheaf} of rank $g-d-1$ living on $J$. Assume $1\leq d\leq g-1$, so that by \cite[Lemma 4.9]{Mukai1} we know that $F$ is simple (that is, $\End_{\O_J}(F) = k$), and
\be\label{tgspacesimple}
\dim \Ext^1_{\O_J}(F,F) = 
\begin{cases}
2g & \textrm{if }C\textrm{ is not hyperelliptic} \\
3g-2 & \textrm{if }C\textrm{ is hyperelliptic}. \\
\end{cases}
\ee
Let $\Spl_J$ be the moduli space of simple coherent sheaves on $J$, and let $M(F)\subset \Spl_J$ be the connected component containing the point corresponding to $F$. It is shown in \cite[Theorem 4.8]{Mukai1} that if $g = 2$ or $C$ is non-hyperelliptic, the morphism
\be\label{mukaimap}
f\colon\widehat J\times J\ra M(F),\quad (\eta,x)\mapsto \mathsf t_x^\ast F\otimes \mathscr P_\eta,
\ee
is an isomorphism. By \eqref{tgspacesimple}, the space $M(F)$ is reduced precisely when $C$ has genus $2$ or is non-hyperelliptic. For $C$ hyperelliptic, $f$ turns out to be an isomorphism onto the reduction $M(F)_{\red}\subsetneq M(F)$, as Mukai showed in \cite[Example 1.15]{Mukai2}. 

\begin{remark}
The moduli space $M(F)$ is a priori only an algebraic space. But an algebraic space is a scheme if and only if its reduction is a scheme. Therefore $M(F)$ is a scheme because of the isomorphism $\widehat J\times J \cong M(F)_{\red}$.
\end{remark}

The following result, which can be seen as a corollary of Theorem \ref{main}, completes the study of Picard sheaves on Jacobians considered by Mukai, namely those of rank $g-d-1$, with $d\leq g-1$.

\begin{teo}\label{cor8183}
Let $C$ be a hyperelliptic curve of genus $g\geq 2$. Let $J$ be its Jacobian and $F$ a Picard sheaf as above. Then, as schemes,
\[
M(F) = \widehat J\times J\times R_g.
\]
\end{teo}

\begin{proof}
The case $g = 2$ is already covered by Mukai's tangent space calculation. By Theorem \ref{main}, it is enough to exhibit an isomorphism $\widehat J\times H\,\widetilde{\ra}\,M(F)$, where as usual $H\subset \Hilb_J$ is the Hilbert scheme component containing the Abel--Jacobi point $[C]$. We will do this by extending the morphism \eqref{mukaimap} defined by Mukai, that is, completing the diagram
\be\label{jdjhaadweq2}
\begin{tikzcd}
\widehat J\times J \arrow{r}{\sim}\arrow[hook]{d} &
M(F)_{\red} \arrow[hook]{d} \\
\widehat J\times H \arrow[dotted]{r}{\phi} & 
M(F)
\end{tikzcd}
\ee
and showing that the extension $\phi$ is an isomorphism. Recall that via the identification $J = H_{\red}$ we can identify a $k$-valued point $x \in J(k)$ with a $k$-valued point of $H$. Also, for any such point $x \in J \subset H$, we will use the notation $x+p_0$ for the point on the Abel--Jacobi curve $\mathsf t_xC \subset J$ obtained by translating $p_0 \in C \subset J$ via the automorphism $\mathsf t_x\colon J\to J$.
Let
\[
\mathcal Z\xhookrightarrow{\iota} H\times J \ra H
\]
be the universal family of the Hilbert scheme: the fibre of $\mathcal Z \to H$ over $\Spec k(x) \hookrightarrow H$ is the subscheme $\mathsf t_xC \subset J$, and $\iota$, the universal Abel--Jacobi map, restricts to $\mathsf{aj}\circ \mathsf t_{-x}\colon\mathsf t_xC\ra C\hookrightarrow \{x\}\times J$ over the point $x \in H$. We now construct a section $\sigma$ of $\mathcal Z\ra H$ restricting to the divisor $dp_0$ on $C$ (in other words: a ``universal'' version of $\xi$). If $q\colon H\ra J$ denotes the projection (forgetting the non-reduced structure) and $u\colon J\ra J$ is the composition $\mathsf t_{dp_0}\circ [d]$, the section $\sigma$ is the map
\[
\sigma\colon H\xrightarrow{(1_H,q)} H\times J\xrightarrow{1_H\times u} H\times J, \quad x\mapsto (x,d(x+p_0)).
\]
Here we view $d(x+p_0)$ as a degree $d$ divisor on the translated Abel--Jacobi curve $\mathsf t_x C \subset J$, in particular the image of $\sigma$ clearly lands inside $\mathcal Z$. 
Let $\mathscr L = \O_{\mathcal Z}(\sigma)$ be the associated line bundle on the total space $\mathcal Z$. 
Then, by construction, restricting $\mathscr L$ to a fibre of $\mathcal Z\ra H$ we get
\be\label{913813}
\mathscr L|_{\mathsf t_xC} = \O_{\mathsf t_xC}(d(x+p_0)) = \mathsf t_{-x}^\ast \xi.
\ee
If we consider the pushforward $\iota_\ast\mathscr L$ to $H\times J$, using Equation \eqref{913813} we obtain
\be\label{ewe323}
(\iota_\ast\mathscr L)|_{x\times J} = 
(\mathsf{aj}\circ \mathsf t_{-x})_\ast (\mathscr L|_{\mathsf t_xC}) = 
\mathsf{aj}_\ast \xi.
\ee
Note that $\mathscr L$ is flat over $H$ (because $\mathcal Z\ra H$ is flat), therefore the same is true for $\iota_\ast\mathscr L$. Since taking the Fourier--Mukai transform commutes with base change, Equation \eqref{ewe323} yields
\be\label{afafferf}
\Phi^1(\iota_\ast \mathscr L)|_{x\times J} = \Phi^1(\mathsf{aj}_\ast \xi) = F.
\ee
Now we consider the following diagram:
 \[ 
 \begin{tikzcd}
 & (\widehat J\times J)\times J\arrow[hook]{d}{i}\arrow{r}{\sim} & 
 (J\times J) \times J\arrow[hook]{d}\arrow{r}{m\times \textrm{id}_J} & 
 J\times J \arrow[hook]{d}\arrow{r}{{\textrm{pr}}_1}
 & J \\
 \widehat J\times J & 
 (\widehat J\times H)\times J\arrow{r}{\sim}\arrow[swap]{l}{\textrm{pr}_{13}} & 
 (J\times H) \times J\arrow{r}{\mu\times \textrm{id}_J} & 
 H\times J & 
 \end{tikzcd}
 \]
where $m$ and $\mu$ are the translation actions by $J$ on $J$ and $H$ respectively. The Fourier--Mukai transform $\Phi^1(\iota_\ast \mathscr L)$ lives on $H\times J$ and is flat over $H$, by flatness of $\iota_\ast\mathscr L$.
By \eqref{afafferf}, we know that the families of sheaves $\Phi^1(\iota_\ast \mathscr L)|_{J\times J}$ and $\textrm{pr}_1^\ast F$ (both flat over $J$) define the same morphism $J\ra M(F)$, namely the constant morphism hitting the point $[F]$. Since Mukai's morphism $\widehat J\times J\ra M(F)$, defined in \eqref{mukaimap}, corresponds (after identifying $J$ with its dual) to the family of sheaves 
\[
(m\times \textrm{id}_J)^\ast \textrm{pr}_1^\ast F\otimes ({\textrm{pr}}_{13}\circ i)^\ast \mathscr P,
\]
it follows that the family 
\[
(\mu\times \textrm{id}_J)^\ast \Phi^1(\iota_\ast\mathscr L)\otimes \textrm{pr}_{13}^\ast \mathscr P
\]
defines an extension $\phi\colon \widehat J \times H\ra M(F)$, completing diagram \eqref{jdjhaadweq2}. We know that $\phi$ is an isomorphism around $[\xi]\mapsto [F]$. Indeed, $\phi$ is precisely the morphism constructed by Mukai in \cite[Prop.~1.12]{Mukai2}, where he proves that $M(\xi)$ and $M(F)$ are isomorphic along a Zariski open subset. The construction is homogeneous, in the sense that $\phi$ does not depend on the initial point $[\xi]\in M(\xi)$. Therefore $\phi$ is globally an isomorphism, as claimed.
\end{proof}

\begin{remark}
The connected component $M(\xi)$ of the moduli space of simple sheaves containing the point $[\xi]$ is the relative Picard variety $\Pic^d(\mathcal Z/H)$, which can be identified with $\widehat J\times H$ by Lemma \ref{lemmaisoab}. It is possible to adapt the proof of \cite[Prop.~1.12]{Mukai2} to show that the birational map 
\[
\Pic^d(\mathcal Z/H)\dashrightarrow M(F)
\]
is everywhere defined (and an isomorphism), giving an immediate proof of Corollary \ref{cor8183}. We preferred to present the argument above, because the construction makes the isomorphism $\phi\colon \widehat J\times H\ra M(F)$ arise directly, as a ``thickening'' of Mukai's isomorphism $\widehat J\times J \ra M(F)_{\red}$. Moreover the argument makes explicit use of (the properties of) the  Fourier--Mukai transform.
\end{remark}

\bigskip
{\noindent{\bf Acknowledgments}.
We are glad to thank Alberto Collino for generously sharing his insight and ideas on the problem. We also thank Martin Gulbrandsen, Micha\l{} Kapustka, Aaron Landesman, Richard Thomas and Filippo Viviani for helpful discussions, and the anonymous referees for suggesting several improvements.}

\clearpage
\bibliographystyle{amsplain}
\bibliography{bib}
\end{document}